\documentclass[11pt]{article}

\RequirePackage[l2tabu, orthodox]{nag}
\usepackage[T1]{fontenc}
\usepackage[utf8]{inputenc}
\usepackage{lmodern}
\usepackage[english]{babel}
\usepackage{color} 
\usepackage[usenames,dvipsnames]{xcolor}
\usepackage{amsrefs,amsthm,amssymb,amsmath}
\usepackage[pdftex]{graphicx}
\usepackage{enumerate}
\usepackage{url}

\def\comment#1{}

\newcommand{\dfcn}[5]{\setlength{\arraycolsep}{1.5pt}\begin{array}{cccc}#1:&#2&\to&#3\\&#4&\mapsto&#5\end{array}}

\newcommand{\R}{\mathbb{R}}
\newcommand{\N}{\mathbb{N}}
\newcommand{\Z}{\mathbb{Z}}
\newcommand{\Q}{\mathbb{Q}}

\newcommand{\T}{\mathbb{S}^1}
\newcommand{\Diff}{\mathsf{Diff}}
\newcommand{\Homeo}{\mathsf{Homeo}}
\newcommand{\homeo}{\mathsf{Homeo}}

\newcommand{\G}{\Gamma}

\newcommand{\PSL}{\mathsf{PSL}}

\newcommand{\SO}{\mathsf{SO}}
\newcommand{\BP}{\mathsf{BP}}

\newcommand{\cH}{\mathcal{H}}
\newcommand{\cG}{\mathcal{G}}

\newcommand{\Xbb}{\mathbb{X}}

\newcommand{\PL}{\mathsf{PL}_+}
\newcommand{\PP}{\mathsf{PP}_+}

\newcommand{\pa}{\mathsf{PDiff}^r(\T)}
\newcommand{\dg}{\mathsf{Diff}^r}

\newcommand{\Isom}{\mathsf{Isom}}

\newcommand{\fw}{\operatorname{FW}}

\newtheorem{thm}{Theorem}[section]
\newtheorem*{thm*}{Theorem}
\newtheorem{claim}{Claim}

\newtheorem{lem}[thm]{Lemma}
\newtheorem{prop}[thm]{Proposition}
\newtheorem{cor}[thm]{Corollary}
\theoremstyle{definition}
\newtheorem{dfn}[thm]{Definition}
\newtheorem{ques}[thm]{Question}

\theoremstyle{remark}
\newtheorem{rem}[thm]{Remark}

\usepackage[margin=2.4cm]{geometry}

\usepackage{indentfirst}
\usepackage{microtype}

\usepackage[colorlinks=true,linkcolor=blue,citecolor=magenta]{hyperref}

\title{Property FW, differentiable structures, and smoothability of singular actions}
\date{}
\author{Yash Lodha, Nicol\'as Matte Bon and Michele Triestino\thanks{Y.L.~was supported by an SNSF \emph{Ambizione} grant. M.T.~was partially supported by PEPS -- Jeunes Chercheur-e-s -- 2017 (CNRS) and Projet ``Jeunes G\'eom\'etres'' of F.~Labourie (financed by the Louis D.~Foundation).}}
\begin{document}
\maketitle

\begin{abstract}
We provide a smoothening criterion for group actions on manifolds by singular diffeomorphisms. We prove that if a countable group $\Gamma$ has the fixed point property FW for walls (e.g.~if it has property $(T)$), every aperiodic action of $\Gamma$  by diffeomorphisms that are of class $C^r$ with countably many singularities  is  conjugate to an action by true diffeomorphisms of class $C^r$ on a homeomorphic (possibly non-diffeomorphic) manifold. As applications, we show that Navas's  result for actions of Kazhdan groups on the circle, as well as  the recent solutions to Zimmer's conjecture, generalise to aperiodic actions by  diffeomorphisms with countably many singularities.  \footnote{\textbf{MSC\textup{2020}:} Primary 22D55, 37C85, 57M60. Secondary 37E10.
}

\end{abstract}


\section{Introduction}

Let $\G$ be a countable group acting by homeomorphisms on a closed differentiable manifold $M$. 
Assume that every $g\in \G$ acts as a diffeomorphism of regularity $C^r$ in restriction to some open set $U_g$ whose complement is countable. We say that $\Gamma$ acts by \emph{countably singular} $C^r$-diffeomorphisms.  For any real $r\in[1,\infty]$ we will denote by $\dg(M)$  the group of $C^r$ diffeomorphisms of a  manifold~$M$, with the convention that when $r$ is not integer, this means functions whose $\lfloor r\rfloor$-th derivative exists and is $(r-\lfloor r\rfloor)$-H\"older continuous. We further let $\Omega \mathsf{Diff}^r(M)$ be the  group of {countably singular} $C^r$-diffeomorphisms. 

We are interested in the following general problem: when is  a singular action $\Gamma \to \Omega \Diff^r(M)$ \emph{smoothable},  i.e.\ topologically conjugate to an action by honest $C^r$-diffeomorphisms? 
 
 In general this need not be the case, and the presence of singularities (even finitely many) can affect substantially the dynamics of an action. To illustrate the matter, consider an action of $\mathbb{Z}$ on $\T$ that is generated by a homeomorphism $f$ which is a $C^{\infty}$-diffeomorphism on the complement of a finite set of points and has no finite orbit. By Denjoy's theorem, if $f$ was conjugate to a $C^2$-diffeomorphism then it would actually be conjugate to a rotation. However, this is not always the case: the first such examples were described by Hall~\cite{Hall}. (Nevertheless, a  result of Yoccoz \cite{yoccoz} shows that Denjoy's theorem holds for real-analytic homeomorphisms whose derivative may vanish at a finite set of points). Several recent papers are devoted to the study of smoothabilty of actions of more general groups on one-manifolds. Well-studied classes of groups of singular diffeomorphisms of the circle include subgroups of the groups  $\PL(\T)$ of \emph{piecewise linear} homeomorphisms of the circle, and of the group $\PP(\T)$ of \emph{piecewise projective} ones.   In this setting, an early result of   Ghys and Sergiescu shows that \emph{Thompson's group} $T$ on the circle is $C^\infty$-smoothable \cite{GS}. In contrast, there exists groups of piecewise linear and projective homeomorphisms that fail to be smoothable, and are even not isomorphic to groups of $C^1$-diffeomorphisms  \cite{BLT, LodhaSimple}. In a different direction, a recent breakthrough  of Kim and Koberda provided examples of finitely generated groups $\Gamma$ of circle diffeomorphisms of every given critical regularity, which is defined as the supremum over $r$ such that $\Gamma$ can be embedded in $\Diff^r(\T)$ \cite{KK}. {In fact, the groups produced by their construction all arise as subgroups of $\Omega \Diff^\infty(\T)$,    providing another instance of how  singular actions can give rise to wild phenomena.}

 The purpose of this paper is to show that the situation is very different provided that the acting group $\Gamma$ has  \emph{Kazhdan's property}  $(T)$ (in fact  the weaker \emph{property} $\fw$, whose definition is recalled below).  In this case,  we obtain a smoothability criterion of broad generality, that holds for actions on closed manifolds of arbitrary dimension.

 
 Recall that a countable group $\Gamma$ is a \emph{Kazhdan group} (or has \emph{Kazhdan's property} $(T)$) if every isometric affine action of $\Gamma$ on a Hilbert space has a fixed point.
 This was not Kazhdan's original definition, but it is equivalent thanks to 
Delorme--Guichardet theorem. (See \cite{kazhdan} as a general reference on property $(T)$.)  
 To recall the  definition of property $\fw$, let~$\G$ be a countable group acting on a set $X$. A subset $A\subset X$ is said to be \emph{commensurated} if  the symmetric difference $g(A) \triangle A$ is finite for every element $g\in \G$. A subset $A\subset X$ is said to be \emph{transfixed} if there exists a $\G$-invariant subset $B\subset X$ such that $A\triangle B$ is finite. A group $\G$ has \emph{property $\fw$} if for every $\G$-action, every commensurated subset is transfixed. Property $\fw$ is named after Barnhill--Chatterji \cite{BC} and has been extensively studied by  Cornulier \cite{Cor-FM,Cor-FWsurvey}  (see also \cite{CanCor} for a recent application to birational dynamics). It is equivalent to a fixed point property for actions on $\operatorname{CAT}(0)$ cube complexes, and admits various other equivalent definitions, for which we refer  to~\cite{Cor-FWsurvey}. Property $(T)$ implies property $\fw$ (see \cite[p.\ 3]{Cor-FWsurvey}).
 

It can happen that  two manifolds $M$ and $N$ are homeomorphic through a  countably singular diffeomorphism, yet they are not diffeomorphic. This goes back to Milnor's discovery of  manifolds homeomorphic to the $7$-sphere, but not diffeomorphic to it  \cite{Mil-7} (manifolds  homeomorphic to spheres are now known as \emph{exotic spheres}). In fact, every exotic sphere  of dimension $d \ge 7$ admits a singular diffeomorphism to $\mathbb{S}^d$ with just one singular point. On the other hand if $M$ and $N$ are as above, then the groups 
$\Omega \mathsf{Diff}^r(M)$ and $\Omega \mathsf{Diff}^r(N)$ are isomorphic.  For this reason, the existence of exotic differentiable  structures should be taken into account in the context studied in this paper.

We show the following result.
%
%
%

\begin{thm} \label{t-countable}
	Let $M$ be a closed differentiable manifold, and fix $r\in[1,\infty]$. Let $\Gamma$ be a finitely generated group with property $\fw$. 
For every homomorphism $\rho\colon \Gamma \to \Omega \mathsf{Diff}^r(M)$,  one of the following holds:
\begin{enumerate}[(i)]
\item \label{i-finite-orbit} the action of $\rho(\Gamma)$ on $M$ has a  finite orbit;
\item \label{i-smooth} there exist a closed differentiable manifold $N$ and a countably singular diffeomorphism $\varphi \colon M \to N$ such that $\varphi\rho(\Gamma)\varphi^{-1}\subset \mathsf{Diff}^r(N) $. \end{enumerate}
\end{thm}

This result is a special case of Theorem~\ref{t-main} below (dealing more generally with singular actions on manifolds endowed with a geometric structure).

\begin{rem} In the second case, the manifold $N$ resulting from the proof is a connected sum of  $M$  with finitely many exotic spheres. In particular, it is homeomorphic to $M$, but not necessarily diffeomorphic. If the dimension $d$ of $M$ is such that the $d$-sphere admits a unique differentiable structure, then it is possible to choose $N=M$ (see Theorem \ref{t-main}). It is known that the differentiable structure on $\mathbb{S}^d$ is unique  for $d\in \{1,2,3,5,6,12,56,61\}$ \cites{Ker-Mil, Wang-Xu, Hill-Hop-Rav}, and never unique for odd  $d\notin \{ 1,3,5,61\}$  \cites{Wang-Xu, Hill-Hop-Rav}, while a general answer is still unknown in even dimension (and is a wide open problem for $d=4$). 
\end{rem}
\begin{rem}
In the particular case of representations of $\Gamma$ into groups of singular diffeomorphisms with \emph{finitely many} singularities, we obtain a more precise statement in the finite-orbit case \eqref{i-finite-orbit}, namely that the action is conjugate to an action that is of class $C^r$ away from a finite invariant subset (Theorem~\ref{t-finitelymany}).
\end{rem}

\medskip

Whenever the group $\Gamma$ has property $\fw$,  Theorem \ref{t-countable}  implies that any result that is valid for a class of regular actions must imply strong restrictions on  a class of singular actions as well.  As a first application of this,  let us  consider again the case $M=\T$. 
Andr\'es Navas proved the following theorem.

\begin{thm}[Navas]
 	\label{t:navas}
 	If $\Gamma$ is a countable Kazhdan group, every homomorphism $\rho\colon \Gamma \to \mathsf{Diff}^{3/2}(\T)$ has finite image.\footnote{More precisely, this result is proven in  \cite{navas(T)} for $C^r$-actions with  $r>3/2$, but can be extended to $C^{3/2}$-actions according to \cite{navas-cocycles}, or also \cite[Remark 5.2.24]{navas-book},  relying on work by Bader--Furman--Gelander--Monod \cite{BFGM}, by means of an $L^p$-analogue of property  introduced by Fisher--Margulis.}
 \end{thm}
(For what concerns actions of lattices in higher rank simple Lie groups, earlier and more precise results were obtained by Ghys, Burger--Monod, and Witte Morris \cite{ghys,burger-monod,witte}.) 
 
  It remains a wide open problem whether  there exists an infinite Kazhdan group of circle {homeomorphisms}. A  special case of this problem  is to determine whether such a group can exists among groups of singular diffeomorphisms. In fact, this problem was open  even for groups of {piecewise linear} and {piecewise projective} homeomorphisms, a case suggested by Navas (see \cite[Remark~5.2.24]{navas-book}, \cite[Question 4.8]{BLT}, and \cite[\S 1]{navas-ICM}). 
  More generally, let us denote by $\pa$ the group of all \emph{piecewise differentiable homeomorphisms} of the circle of class $C^r$, i.e.~the subgroup of $\Omega \mathsf{Diff}^r(M)$ consisting of elements for which the singular set is finite, and moreover at every singular point the right and left derivatives exist up to order $r$.
The combination of Theorem~\ref{t-countable} with Navas's theorem implies the following.
\begin{cor}\label{c-T}
 Let $\Gamma$ be a countable Kazhdan group. Then the following hold.
 \begin{enumerate}[(i)]
 \item  \label{i-T} Every homomorphism $\rho\colon \Gamma\to \mathsf{PDiff}^{3/2}(\T)$ has finite image.

\label{c-T-countable}
\item \label{i-T-countable}
 For every homomorphism $ \rho \colon \Gamma \to \Omega\mathsf{Diff}^{3/2}(\T)$, the action of $\rho(\Gamma)$ on $\T$ has a finite orbit.  
\end{enumerate}
\end{cor}

In an independent preprint \cite{CorPL}, Yves de Cornulier  has  obtained part \eqref{i-T} of Corollary \ref{c-T} for actions by {piecewise linear} or {projective} homeomorphisms where also points of discontinuity are allowed (see Question~\ref{q:exotic} below and the remark following it).

\begin{rem}
	The relevant difference between parts \eqref{i-T} and \eqref{i-T-countable} in  Corollary \ref{c-T} is that Thurston's stability theorem \cite{Th} holds in the group $\mathsf{PDiff}^{1}(\T)$ but not in the group $\Omega\mathsf{Diff}^{1}(\T)$. 
\end{rem}

\begin{rem}\label{r:sqrt2}
Corollary \ref{c-T} cannot be strengthened to obtain that every countable subgroup of $\pa$ with property $\fw$ is finite. For example the group $\PSL(2,\mathbb Z[\sqrt{2}])$ acts on the circle by M\"obius transformations. This group has property $\fw$, as proved by Cornulier  relying on work of Haglund and  Carter--Keller  \cite[Example 6.A.8]{Cor-FWsurvey}.
\end{rem}
\begin{rem}
If we allow an uncountable closed set of singularities (of empty interior), then \emph{every} countable group of circle homeomorphisms acting minimally is isomorphic to a group of singular diffeomorphisms. Blowing up an orbit, it is possible to obtain any desired level of regularity on the complement of a Cantor set. \end{rem}

Let us now discuss an application in higher dimension. Theorem \ref{t-main} can be applied to lattices in higher rank simple Lie groups, which are the primary source of examples of Kazhdan groups. The study of the rigidity properties of actions of such groups on compact manifolds was proposed by Zimmer \cite{Zim-program} as a ``non-linear'' generalisation of the classical rigidity results of  Mostow and Margulis, and is a well-developed topic. Most results and methods that have been developed  in this setting apply to actions by diffeomorphisms (with some required degree of regularity).  A central conjecture in Zimmer's program states that a lattice in a higher rank simple Lie group  has only (virtually) trivial actions on closed manifolds of dimension $<d$, where $d$ is an explicit constant depending on the ambient Lie group (bounded below by its real rank).  The conjecture has been (partially) solved recently with the breakthrough work of Brown, Fisher and Hurtado \cite{BFH,BFH2,DZ} for a class of actions by $C^2$-diffeomorphisms.

The case of actions by homeomorphisms that are not regular is considerably less understood, and the techniques employed come mostly from algebraic topology rather than from dynamical systems, see the survey \cite{Weinberger}. In particular, it remains  a well-known open problem whether Zimmer's conjecture holds for all action by homeomorphisms. 
In combination with the aforementioned results of Brown, Fisher and Hurtado \cite{BFH,BFH2}, Theorem~\ref{t-countable}  yields the following for countably singular actions. 

\begin{cor} \label{c-Zimmer}
	Let $M^d$ be a closed manifold of dimension $d$.
	Let $G$ be a connected Lie group, whose Lie algebra is simple and with finite centre. Assume that the real rank of $G$ is $r>d$ and let $\Gamma\subset G$ be a cocompact lattice, or $\Gamma=\mathsf{SL}(r+1,\Z)$.
	For any morphism $\rho:\Gamma\to \Omega\mathsf{Diff}^2(M)$, the action of $\rho(\Gamma)$
	on $M$ has a finite orbit.
\end{cor}
Using a more general version of Theorem \ref{t-countable} (namely Theorem \ref{t-main}), we also establish an analogous statement in the case of actions preserving a volume form, see Corollary  \ref{c-Zimmer-vol}.

\begin{rem}
	Here we have simplified the statement in order to avoid introducing too much
	cumbersome notation. The interested reader will be able to improve the above statement from the results appearing in \cite{BFH,BFH2,DZ,Cantat}. Note also that Theorem~\ref{t-finitelymany} gives a stronger result for actions by elements with finitely many singularities.
\end{rem}

\begin{rem}
	For {groups of elementary matrices}, and for {automorphism groups of free groups}  (and some relatives) questions related to Zimmer's program have been largely investigated by Ye, with no restrictions on regularity (see e.g.~\cite{Ye1,Ye2,Ye4}).	Let us also mention \cite{FS}, where it is proved that groups with no finite quotients and a very strong fixed point property (as some models of \emph{random groups}), cannot act on a closed manifold preserving a smooth volume form. It is a well-known open problem, attributed to Gromov, to prove that a random group cannot act on a closed manifold.
\end{rem}

Let us give an outline of the paper and of the proof of Theorem \ref{t-countable}. The core of our argument has purely topological dynamical nature (no manifolds are involved, only group actions on locally compact spaces), and is contained in Section \ref{s:groupoids}. We consider an action of a countable group~$\Gamma$ on a locally compact space $X$ whose \emph{germs} at every point  belong to a given \emph{groupoid of germs}~$\cH$, except for a  finite set of singular points for every  group element (all the necessary terminology will be defined in Section~\ref{s:groupoids}). We provide a criterion on $\cH$ that implies, when $\Gamma$ has property $\fw$, that the action is conjugate to an action whose germ at every point belongs to $\cH$.  To do so we construct a commensurating action of $\Gamma$  inspired from work of Juschenko--Nekrashevych--de la Salle on amenability of groups of homeomorphisms \cite{JNS}.   We then extend this to the case of countably many singularities by transfinite induction on the \emph{Cantor--Bendixson rank} of the set of singular points of the generators of $\Gamma$.

In Section \ref{s:manifolds} we apply this criterion to singular actions on manifolds. In the course of the proof, we are faced to the problem of extending a finite collection of  germs of singular diffeomorphisms to the whole ambient manifold,  without introducing new singularities. Such an extension does not always exist. In the case of actions by diffeomorphisms the obstruction to its existence is measured by the group of twisted $n$-spheres $\Gamma_n=\pi_0(\mathsf{Diff}_+(\mathbb S^{n-1}))/\pi_0(\mathsf{Diff}_+(\mathbb D^n))$, which is in one-to-one correspondence with the differentiable structures of $\mathbb{S}^n$ (for $n\neq 4$). Taking a connected sum with exotic spheres allows to eliminate this obstruction by producing another manifold, homeomorphic to the original one, on which the action can be regularised.  In fact, the set of diffeomorphisms of class $C^r$ do not play any special role in the proof of Theorem \ref{t-countable}. We shall work in a more general setting  (see Theorem \ref{t-main}), by allowing actions that are countably singular with respect to an arbitrary given {groupoid of germs}  $\cG$ of local homeomorphisms of $\mathbb{R}^d$, provided $M$ has the structure of a  $\cG$-\emph{manifold}, i.e.~admits an atlas whose changes of charts have germs in $\cG$ (see Thurston's book \cite[Chapter 3]{Th-book} as a standard reference on $\cG$-manifolds).

In Section \ref{s:circle} we discuss the case of piecewise differentiable circle diffeomorphisms, and then  focus on the example of  groups of \emph{piecewise linear} (PL) homeomorphisms of the circle. 
In this case, we provide a more elementary proof of the fact that they do not have property $(T)$, which only relies on H\"older's theorem. Namely, we will exhibit an isometric action of $\PL(\T)$ on the Hilbert space $\ell^2(\T)$, with linear part defined by the action of the group on $\T$: we twist the linear action with a cocycle that measures the failure of elements to be affine. This cocycle has been widely used (implicitly or explicitly) for understanding many properties of groups of piecewise linear homeomorphisms \cite{GS,minakawa,GL,Liousse} (just to cite a few). This is to be compared to the result by Farley and Hughes \cite{farley, hughes} that Thompson's group $V$ has the Haagerup property, which is also proved by exhibiting an explicit proper action of $V$ on a Hilbert space (Farley's  proof of this result in \cite{farley} had a gap and has been fixed by  Hughes \cite{hughes}).

Let us end this introduction with a list of problems.

\begin{ques}
Which subgroups of $\pa$ have the Haagerup property?
\end{ques}

\begin{ques} 
Does the group $\Omega \dg(\R)$ of countably singular diffeomorphisms of the real line contain infinite countable property $(T)$ subgroups? 
\end{ques}

\begin{ques}
Study Zimmer's conjecture for piecewise differentiable (or PL) actions of lattices in higher rank simple Lie groups on manifolds of small dimension. (cf.~the work of Ye \cite{Ye3}.)
\end{ques}

\begin{ques}
Consider the renormalised linear action of $\mathsf{SL}(n+1, \Z)$ on the sphere $\mathbb{S}^n$, with $n\ge 7$. For which differentiable structures on $\mathbb{S}^n$ is this action topologically conjugate to a smooth action?
\end{ques}
 
\begin{ques}\label{q:AEIT}
Extend Corollary \ref{c-T} to groups of discontinuous transformations, such as AIET, the group of affine interval exchange transformations. (The group of isometric interval exchange transformations does not contain infinite countable Kazhdan subgroups \cite{DFG}.) 
\end{ques}
\begin{ques}[cf.~\cite{Sergiescu}]\label{q:exotic}
	{Classify the subgroups of the group of piecewise projective circle homeomorphisms $\PP(\T)$ that are topologically conjugate to $\PSL(2,\R)$, up to PP conjugacy.}
\end{ques}
 Questions~\ref{q:AEIT} and \ref{q:exotic} (that were written in a preliminary version of this paper) have been meanwhile answered by Yves de Cornulier. In relation to Question \ref{q:AEIT}, he proves \cite[Corollary~1.8]{CorPL} that the group AIET contains no infinite countable subgroup with property $\fw$ (and more generally his arguments allow to consider transformations with finitely many discontinuity points; see also Remark~\ref{r-YvesCornulier}). On the other hand, the answer \cite[Theorem~1.13]{CorPL} to Question \ref{q:exotic} makes use of the classification of  projective one-manifolds and their automorphism groups: using the fact that $\PSL(2,\Z[\sqrt{2}])$ has property $\fw$ (Remark~\ref{r:sqrt2}) and is dense in $\PSL(2, \R)$, he concludes that the action of a topological $\PSL(2,\R)$ must preserve a projective structure on $\T$.

 \section{Property $\fw$ and singular actions on locally compact spaces }\label{s:groupoids}

\subsection{Groupoids of germs}

Throughout the section, we  let $X$ be a locally compact Hausdorff space.   We will be primarily interested in the case where the space $X$ is compact. However, in the course of the proofs, it will be useful to consider non-compact spaces as well.

Recall that the \emph{support} of a homeomorphism $h$ of $X$ is the subset $\mathsf{supp}(h)=\overline{\{x\in X\colon h(x)\neq x\}}$ of $X$, which is the closure of the set of points which are moved by $h$. A  \emph{germ} (on  the space $X$) is the equivalence class of a pair $(h, x)$, where $x\in X$ and $h$ is a homeomorphism defined from a neighbourhood of $x$ to a neighbourhood of $h(x)\in X$, and where $(h_1, x_1)$ and $(h_2, x_2)$ are equivalent if $x_1=x_2$ and $h_1$ and $h_2$ coincide on a neighbourhood of $x_1$. We denote by $[h]_x$ the equivalence class of $(h, x)$.  If $\gamma=[h]_x$ is a germ, where  $(h, x)$ is a representative pair of $\gamma$, the points $s(\gamma)=x$ and $t(\gamma)=h(x)$ are called the \emph{source} and the \emph{target} of $\gamma$.  Two germs $\gamma_1$ and  $\gamma_2$ can be multiplied provided $t(\gamma_2)=s(\gamma_1)$, and in this case $\gamma_1\gamma_2=[h_1h_2]_{s(\gamma_2)}$, where $(h_1, s(\gamma_1))$ and $(h_2, s(\gamma_2))$ are any choice of representatives of $\gamma_1$ and $\gamma_2$. They are inverted according to the rule $[h]_x^{-1}=[h^{-1}]_{h(x)}$.   Given a point $x\in X$, we will denote by $1_x$ the germ of the identity homeomorphism at $x$.

\begin{dfn} \label{d-groupoid}A \emph{groupoid of germs} (over $X$) is a set $\cG$ of germs on $X$ which verifies the following properties:
\begin{enumerate}[(i)]
\item if $\gamma_1$ and $\gamma_2\in \cG$ are such that $s(\gamma_1)=t(\gamma_2)$, then $\gamma_1\gamma_2\in \cG$, and we have $\gamma^{-1}\in \cG$ for every $\gamma\in \cG$;
\item we have $1_x\in \cG$ for every $x\in X$;
\item \label{i-groupoid-etale} every $\gamma\in \cG$ admits a representative pair $(h, s(\gamma))$ with the property that $[h]_y\in\cG$ for every $y$ in the domain of definition of $h$. 
\end{enumerate}
The  \emph{topological full group}  of $\cG$ is the group $F(\cG)$ of all compactly supported homeomorphisms $g$ of the space $X$ with the property that $[g]_x\in \cG$ for every point $x\in X$.
\end{dfn} 

\subsection{Finitely many singularities}

In this subsection we study group actions on locally compact spaces whose germs at every point belong to a given groupoid of germs, except for  finitely many isolated ``singularities''. To formalise this idea, we will work in the following setting. 

\begin{dfn} Let $\cH$ be a groupoid of germs over $X$. We denote by $S\cH$ the groupoid of germs consisting of all germs $\gamma$ that admit a representative  $(h, s(\gamma))$ such that  $[h]_y\in \cH$ for all $y\neq s(\gamma)$ in the domain of definition of $h$. 

A \emph{singular pair} of groupoids is a pair $\cH\subset \cG$ of groupoids of germs over $X$ such that $\cH\subset \cG\subset S\cH$. 
\end{dfn}

Given a singular pair $\cH\subset \cG$ will say that a germ $\gamma \in \cG$ is \emph{singular} if it does not belong to $\cH$, and that it is \emph{regular} otherwise.  Similarly given $g\in F(\cG)$, we will say that a point $x\in X$ is \emph{singular} for $g$ if $[g]_x\notin \cG$, and \emph{regular} otherwise.

A simple compactness argument yields the following lemma:
\begin{lem}\label{l-co-discrete}
 If $\cH\subset \cG$ is a singular pair, every element $g\in F(\cG)$ has finitely many singular points. 
\end{lem}
\begin{proof}
This follows from the fact that elements of $F(S\cH)$ are compactly supported, and from a simple compactness argument. 
\end{proof}

We also introduce the following \emph{ad-hoc} terminology.
\begin{dfn}\label{d-extension}
 We say that the a singular pair $\cH\subset \cG$ has  \emph{resolvable singularities} if the following holds.  For every choice of finitely many germs $\gamma_1,\ldots, \gamma_\ell\in \cG$ such that $t(\gamma_1), \ldots, t(\gamma_\ell)$ are distinct points,  and every compact subset $K\subset X$  there exists an element $\varphi\in F(\cG)$ with the following properties.
\begin{enumerate}[(i)]
\item We have $[\varphi]_{t(\gamma_i)}\gamma_i\in \cH$ for every $i=1,\ldots, \ell$.  \label{i-resolvable}
\item We have $[\varphi]_y\in \cH$ for every $y\in K\setminus \{t(\gamma_1), \ldots, t(\gamma_\ell)\}$.\label{ii-resolvable}    \end{enumerate}
\end{dfn}
Intuitively, this means that any finite family of singular germs  $\gamma_1, \ldots, \gamma_\ell\in\cG$ can be ``resolved'' (i.e.~brought back to $\cH$) by post-composing them with an element $\varphi\in F(\cG)$ that can be chosen without any additional singularities in any arbitrarily large compact subset. 
Note that if the space $X$ is compact, the compact  subset $K$ is redundant in the definition as we can  choose $K=X$.

The reader can have in mind the following example: $\cH$ is the groupoid of all germs of partially defined diffeomorphisms of the circle, and $\cG$ is defined similarly by allowing isolated discontinuities for the derivatives. In this situation, the fact that the pair $\cH\subset \cG$ has resolvable singularities  is easy to verify (and will follow from Proposition~\ref{l-Dr}).

We consider now subgroups of the topological full group $F(\cG)$.

\begin{rem}
	If a group $\Gamma$ is countable, then property $\fw$ implies automatically that it is finitely generated (this is already a consequence of property FA of Serre, which is implied by $\fw$, see \cite{Cor-FWsurvey}). 
\end{rem}

Given a subgroup $\Gamma\subset F(\cG)$, its \emph{support} $\mathsf{supp}(\Gamma)=\overline{\bigcup_{h\in \Gamma}\mathsf{supp}(h)}$ is the closure in $X$ of the set of points of $X$ that are moved by some element of $\Gamma$. When the subgroup is finitely generated, as it is the case for countable subgroups with property $\fw$, the support of $\Gamma$ is just the union of the supports of elements in a finite generating system. 
We also say that a point $x\in X$ is \emph{singular} for $\Gamma$ if there exists an element $g\in \Gamma$ such that $x$ is singular for $g$.

\begin{prop}\label{p-pair}
Let $\cH\subset \cG$ be a singular pair of groupoids of germs over  $X$ with resolvable singularities. Let $\G\subset F(\cG)$ be a countable subgroup, with property $\fw$. Then there exists a homeomorphism $\varphi\in F(\cG)$  such that  the set of singular points of the conjugate group $\varphi \Gamma \varphi^{-1}$ is finite and consists of points with a finite  $\varphi\Gamma \varphi^{-1}$-orbit.

In particular, if the action of $\Gamma$ on its support has no finite orbits, then $\Gamma$ is conjugate in $F(\cG)$ to a subgroup of $F(\cH)$.
\end{prop}

 In the proof, we will use the following terminology. Given a pair of groupoids of germs $\cH\subset \cG$, the corresponding \emph{coset space} $\cG/\cH$ is defined as the set of equivalence classes of  the equivalence relation on $\cG$ that identifies two elements $\gamma_1$ and $\gamma_2$ of $\cG$  if $t(\gamma_1)=t(\gamma_2)$ and $\gamma_1^{-1}\gamma_2\in \cH$. The equivalence class of $\gamma$ is denoted by $\gamma\cH$. Note that the target map $t\colon \cG\to X$ descends to a well-defined map $t\colon \cG/\cH\to X$, $t(\gamma\cH)=t(\gamma)$.

\begin{proof}[Proof of Proposition \ref{p-pair}]
 We let the group $\G$ act on the coset space $\cG/\cH$ by the rule 
 \[g \cdot \gamma\cH = [g]_{t(\gamma)}\gamma \cH, \quad  \text{where }g\in \G \text{ and }
\,  \gamma \cH\in \cG/\cH.\]
 Observe that the target map $t\colon \cG/\cH\to X$ is $\Gamma$-equivariant. Let $Y=\mathsf{supp}(\Gamma)\subset X$ be the support of $\Gamma$. The action of $\Gamma$ is trivial in restriction to fibres $t^{-1}(x)$ for $x\notin Y$. Therefore, the subset $t^{-1}(Y)=\{\gamma\cH\in \cG/\cH \colon t(\gamma)\in Y\}$ is a $\Gamma$-invariant subset of $\cG/\cH$. Consider  the trivial section $A=\{1_x\cH \colon x\in Y\}\subset t^{-1}(Y)$ in $\cG/\cH$.  
 
 \begin{claim}
 The subset $A$ is commensurated.
 \end{claim} 
 \begin{proof}[Proof of claim]
Take an element $g\in \G$. For every coset $1_x\cH\in A$, the condition that $g \cdot  1_x\cH=[g]_x\cH\notin A$ is equivalent to the fact that $[g]_x\notin \cH$, and there are only finitely many points $x\in Y$ with this property by Lemma~\ref{l-co-discrete}. This shows that $g(A)\setminus A$ is finite. The same reasoning applied to $g^{-1}$ shows that $A\setminus g(A)$ is also finite. Hence $A$ is commensurated. 
 \end{proof}

  Since we are assuming that $\G$ has $\fw$, it follows that $A$ is transfixed. Let $B\subset t^{-1}(Y)$ in $\cG/\cH$, be a $\G$-invariant subset  such that $A\triangle B$ is finite.

\begin{claim}
There exists a finite subset $E\subset Y$ which is $\Gamma$-invariant and such that for every $x\in Y\setminus E$, we have $|t^{-1}(x)\cap B|=1$.
\end{claim}
\begin{proof}[Proof of claim]
We set $E=\{x\in Y\colon |t^{-1}(x)\cap B|\neq 1\}$. 
Then, by invariance of $B$ and equivariance of~$t$, it follows that $E$ is $\Gamma$-invariant. Since $B\triangle A$ is finite, and $A\cap t^{-1}(x)=\{1_x\cH\}$ for every $x\in Y$,  we deduce that $E$ must be finite.
\end{proof}

For every $x\in Y\setminus E$, denote by $\gamma_x\cH$ the unique element of $t^{-1}(x)\cap B$  (where we fix arbitrarily a representative $\gamma_x$ of the coset for every $x\in Y\setminus E$).  Note that, since  $B\triangle A$ is finite, we have $\gamma_x\cH=1_x\cH$ for all but finitely many $x$. Note also that $t(\gamma_x)=x$ for every $x\in Y$. Invariance of $B$  reads as follows:
\begin{equation}
\label{e-inv} \gamma_{g(x)}\cH=[g]_x\gamma_x\cH, \quad \forall x\in Y\setminus E,\,  \forall g\in \G.
\end{equation}
\begin{claim}\label{c-conj}
There exists $\varphi \in F(\cG)$ with the property that
\begin{equation*}
\label{e-phi} [\varphi^{-1}]_{\varphi(x)}\cH=\gamma_x\cH, \quad \forall x\in Y\setminus E.
\end{equation*}\end{claim}
\begin{proof}[Proof of claim]
Let $\Sigma\subset Y\setminus E$ be the finite subset of points $x\in Y\setminus E$ such that $\gamma_x\cH\neq 1_x\cH$. We  use the assumption that the pair $\cH\subset \cG$ has resolvable singularities, applied  to the finite collection of germs $\{\gamma_x\colon x\in \Sigma\}$ and to the compact subset $K=Y$. We obtain that there exists $\varphi\in F(\cG)$ such that

\begin{enumerate}[(a)]
\item \label{i-germs}for every $x\in \Sigma$  we have  $[\varphi]_x\gamma_x \in \cH$,

\item \label{i-elsewhere}  for every $x\in Y\setminus \Sigma$ we have  $[\varphi]_x\in \cH$.
\end{enumerate}
Property \eqref{i-germs} can be rewritten as $[\varphi^{-1}]_{\varphi(x)}\cH=\gamma_x\cH$ for $x\in \Sigma$. Moreover if $x\in Y\setminus (E\cup \Sigma)$, property   \eqref{i-elsewhere} implies that $[\varphi^{-1}]_{\varphi(x)}\cH=1_x\cH$, which is equal to $\gamma_x\cH$ since $x\notin \Sigma$. This concludes the proof of the claim. 
\end{proof}

\begin{claim}
All singular points of  $\varphi\G \varphi^{-1}$ are contained in $\varphi(E)$. 
\end{claim}
\begin{proof}[Proof of claim]
Let $y\notin \varphi(E)$ and let us show that for every $g\in \G$ the germ of $\varphi g \varphi^{-1}$ at $y$ belongs to $\cH$. Since $\varphi \Gamma \varphi^{-1}$ is supported on $\varphi(Y)$, we can assume that $y=\varphi(x)$ for some $x\in Y\setminus E$ and the claim is equivalent to:

\[[\varphi]_{g(x)} [g]_{x}[\varphi^{-1}]_{\varphi(x)}\in \cH,\quad \forall x\in Y\setminus E,\]
which is equivalent to
\[[\varphi^{-1}]_{\varphi (g(x))}\cH=[g]_{x}[\varphi^{-1}]_{\varphi(x)} \cH,\]
which by Claim \ref{c-conj} is equivalent to 
\[\gamma_{g(x)}\cH=[g]_{x}\gamma_x\cH\]
which is exactly \eqref{e-inv}. 
\end{proof}
It follows that all the singular points of $\varphi \Gamma \varphi^{-1}$ are contained in the finite $\varphi \Gamma \varphi^{-1}$-invariant subset $\varphi(E)$. This concludes the proof of Proposition \ref{p-pair}. \end{proof}

\setcounter{claim}{0}

\subsection{Countably many singularities}

The purpose of the remainder of this section is to study a more general situation where we allow group actions with a \emph{countable} set of singularities with respect to a given groupoid of germs $\cH$. We work in the following setting. 
\begin{dfn}
Let $\cH$ be a groupoid of germs over $X$. We let $\Omega\cH$ be the groupoid of all germs $\gamma$ that admit a representative $(h, s(\gamma))$ with the property that $[h]_y\in \cH$ for all but at most countably many points $y$ in the domain of definition of $h$. 

\end{dfn}
It is easy to check that $\Omega\cH$ is a well-defined groupoid of germs.  As before, given an element $g\in F(\Omega \cH)$, we will say that a point $x\in X$ is \emph{singular} for $g$ if $[g]_x\notin\cH$.

\begin{lem}\label{l-singular-countable}
For every element $g\in F(\Omega\cH)$, the set of singular points of $g$ is countable and compact.
\end{lem}
\begin{proof}
Recall that every element $g\in F(\Omega \cH)$ has compact support, by the definition we use of the full group. So the set of singular points in contained in the compact subset $\mathsf{supp}(g)$. Countability follows from a  compactness argument. The fact that the complement of the set of singular points of $g$ is open is a consequence of condition \eqref{i-groupoid-etale} in Definition \ref{d-groupoid}. 
\end{proof}

 Our aim is to prove a result analogous to  Proposition \ref{p-pair} that applies to the pair $\cH\subset \Omega \cH$. In order to state it we first need to give an alternative point of view on the groupoid $\Omega\cH$ via the notion of \emph{Cantor--Bendixson rank} of a compact space (see \cite[\S 6]{Kechris}). Let us recall here its definition in the special case of \emph{countable} compact spaces, that will be enough for our purposes. Let $C$ be a countable compact space. If such as space is non-empty, then it must contain isolated points by Baire's theorem.  Its \emph{Cantor--Bendixson derivative}, denoted $C'$,   is defined as the complement of the subset of isolated points in $C$. Given a countable ordinal $\alpha$   define  a compact countable subset $C_\alpha\subset C$ by transfinite induction by setting $C_0= C$, and
\[C_\alpha=\left\{\begin{array}{lr}C_{\beta}', & \text{ if } \alpha=\beta+1, \\
\\
 \bigcap_{\beta<\alpha} C_\beta, & \text{ if }  \alpha \text{ is a limit ordinal}. \end{array}\right.\]
The family $C_\alpha$ is a decreasing ordered family of closed subsets of $C$. Using the fact that $C$ is countable, it follows that there exists a smallest countable ordinal $\rho$  such that $C_\rho=\varnothing$. The ordinal~$\rho$ is called the \emph{Cantor--Bendixson rank} of $C$. Note also that, as a consequence of compactness of~$C$, the ordinal $\rho$ is of the form $\rho=\beta+1$ and   $C_\beta$ is a finite subset (the reader should be warned that sometimes its predecessor $\beta$ is called the Cantor--Bendixson rank of $C$).

For every countable ordinal~$\alpha$ we define a groupoid $\cH_\alpha$ by setting:

 \[\cH_\alpha=\left\{\begin{array}{lr}S\cH_\beta, & \text{ if } \alpha=\beta+1, \\
 \\ S(\bigcup_{\beta<\alpha} \cH_\beta), & \text{ if }  \alpha \text{ is a limit ordinal}. \end{array}\right.\]
 
 \begin{lem}\label{l:singular_rank}
 Let $g\in F(\Omega\cH)$, and let $C\subset X$ be the set of singular points of $g$. Let $\rho$ be the Cantor--Bendixson rank of $C$. Then $g\in F(\cH_\rho)$. 
 \end{lem}
Note that $C$ is countable and compact by Lemma \ref{l-singular-countable}.
\begin{proof}[Proof of Lemma \ref{l:singular_rank}]
Let $C_\alpha$ be the transfinite sequence as in the definition of Cantor--Bendixson rank. We show by induction on $\alpha\le \rho$ that if $x\in C\setminus C_\alpha$ then $[g]_x\in \cH_\alpha$. First assume that $\alpha=1$. Then $x\in C\setminus C_1$ is an isolated point in $C$, and therefore $x$ has a neighbourhood $U$ in $X$  that does not intersect $C$, i.e.~such that $[g]_y\in \cH$ for every $y\in U$, with $y\neq x$. It follows that $[g]_x\in S\cH=\cH_1$. For the induction step, assume first that $\alpha$ is a limit ordinal. This implies that $x\in C\setminus C_\beta$ for some $\beta<\alpha$ and by the inductive hypothesis we have $[g]_x\in \cH_\beta\subset \cH_\alpha$. Assume now that $\alpha=\beta+1$ is a successor. We can assume that $x\in C_\beta\setminus C_\alpha$, or we  are done again by the inductive hypothesis. Note that $C_\beta\setminus C_\alpha$ consists of isolated points in $C_\beta$, and therefore we can find a neighbourhood $U$ of $x$ such that $(U\setminus \{x\})\cap C\subset (C\setminus C_\beta)$ and it follows by inductive hypothesis that for every $y\in U\cap C$, with $y\neq x$, we have $[g]_y\in \cH_\beta$, from which we conclude that $[g]_x\in \cH_\alpha=S\cH_\beta$.
\end{proof}

\begin{dfn}\label{d-finite-extension} We say that $\cH$ has \emph{countably resolvable singularities} if for every countable ordinal $\alpha$ the following holds:
\begin{enumerate}[(i)]
\item if $\alpha=\beta+1$ is a successor, then the pair $\cH_\beta\subset \cH_\alpha$ has resolvable singularities;
\item if $\alpha$ is a limit ordinal, then the pair $\bigcup_{\beta< \alpha}\cH_\beta  \subset \cH_\alpha$ has resolvable singularities. 
\end{enumerate}

\end{dfn}
Once again, the reader can have in mind the example of the groupoid of germs of all partially defined diffeomorphisms of the circle (Proposition~\ref{l-Dr}). We are now ready to state and prove the final result of this section:

\begin{prop}\label{p-countable}
Let $\cH$ be a groupoid of germs over  $X$ with countably resolvable singularities. Let $\G\subset F(\Omega\cH)$ be a finitely generated subgroup with property $\fw$. Then one of the following possibilities holds. 
\begin{enumerate}[(i)]
\item \label{i-countable-finite-orbit} The action of the group $\G$ on its support has a finite orbit.
\item \label{i-countable-conjugate} The group $\G$ is conjugate in $F(\Omega\cH)$ to a subgroup of $F(\cH)$.
\end{enumerate}
\end{prop}

\begin{proof}
We assume that \eqref{i-countable-finite-orbit} does not hold and show that \eqref{i-countable-conjugate} holds. Note that the support of the action of $\Gamma$ is compact (contained in the union of the supports of its generators). Using Lemma~\ref{l:singular_rank} we obtain that $\Gamma\subset F(\cH_\rho)$ for some countable ordinal $\rho$ (it is enough to take the largest such ordinal over a finite generating subset of~$\G$). 
We show  by induction on $\rho$ that $\Gamma$ is conjugate in $F(\cH_\rho)$ to a subgroup of $F(\cH)$. Assume first that $\rho=\beta+1$ is a successor, so that $\cH_\rho=S\cH_\beta$. Then the pair $\cH_\beta\subset \cH_\rho$ is a singular pair by definition, and it has resolvable singularities by Definition~\ref{d-finite-extension}. As we are assuming that the action of $\Gamma$ has no finite orbit in its support, Proposition~\ref{p-pair}  yields that $\Gamma$ is conjugate in $F(\cH_\rho)$ to a subgroup of $F(\cH_\beta)$, and since the assumption that $\Gamma$ does not have finite orbits in its support is invariant under conjugacy,  we are reduced to the inductive hypothesis. Next, assume that $\rho$ is a limit ordinal, so that $\cH_\rho= S(\bigcup_{\beta<\rho} \cH_\beta)$. The pair $\bigcup_{\beta<\rho} \cH_\beta\subset \cH_\rho$ is  a singular pair and has resolvable singularities, and using again Proposition \ref{p-pair} we obtain that there exists $\varphi\in F(\cH_\rho)$ such that $\varphi \Gamma \varphi^{-1}\subset F(\bigcup_{\beta<\rho} \cH_\beta)=\bigcup_{\beta<\rho}F( \cH_\beta)$. Using again that $\varphi \G \varphi^{-1}$ is finitely generated, it cannot be written as an increasing union of proper subgroups. Therefore it must be contained in $F(\cH_\beta)$  for some $\beta<\rho$. By the inductive hypothesis, this concludes the proof of the proposition. \qedhere
\end{proof}

\section{Singular actions on manifolds} \label{s:manifolds}

Throughout the section we fix $d\ge 1$ and let $\cG$ be a groupoid of germs over $\mathbb{R}^d$. Let $M$ be a topological manifold of dimension $d$ (all topological manifolds will be assumed to be Hausdorff and second countable). We will say that $M$ is a $\cG$-\emph{manifold} if it is  endowed with an atlas such that the corresponding changes of charts have all germs in the groupoid $\cG$.   See Thurston's book \cite[Chapter 3]{Th-book} for more details. 
Whenever $M$ and $N$ are $\cG$-manifolds, a homeomorphism $\varphi\colon M\to N$ between $\cG$-manifolds will be called a $\cG$-\emph{homeomorphism} if in coordinate charts, all the germs of $\varphi$ belong to $\cG$. An \emph{$\Omega\cG$-homeomorphism} (or a \emph{countably singular} $\cG$-homeomorphism) $\varphi \colon M\to N$ between $\cG$-manifolds is a homeomorphism of topological manifolds which is moreover a $\cG$-homeomorphism in restriction to an open subset  of $M$ whose complement is countable. We will denote by $\homeo_\cG(M)$ the group of self $\cG$-homeomorphism of  a $\cG$-manifold $M$, and by $\Omega\homeo_\cG(M)$ the group of countably singular ones.

As the main example, fix $r\ge 1$ and let $\mathcal{D}^r$  be the groupoid of germs of all local diffeomorphisms of class $C^r$ of $\mathbb{R}^d$. Then a  $\mathcal{D}^r$-manifold $M$ is simply a $C^r$-differentiable manifold, and a $\mathcal{D}^r$-homeomorphism is a diffeomorphism of class $C^r$.   By a classical theorem of Whitney \cite{Whit}, every  $C^r$-differentiable manifold admits a unique compatible $C^\infty$-differentiable structure up to diffeomorphism, therefore in this case we simply say that $M$ is a {differentiable manifold}. We will say that a differentiable manifold  is an \emph{exotic sphere} if it is homeomorphic to a euclidean sphere $\mathbb{S}^d$ (in particular, we use the convention that the standard euclidean sphere is itself an exotic sphere).

As another example that will be relevant, let $\mathcal{D}^r_{\mathsf{Leb}}$ be the groupoid of germs of $C^r$-diffeomorphisms of $\mathbb{R}^d$ whose Jacobian at every point has determinant $1$ (that is, those that  preserve the Lebesgue volume form $dx_1\wedge\cdots \wedge dx_d$). Then a $\cG$-manifold $M$ is an oriented  $C^r$-manifold endowed with a volume form $\omega$ (see \cite[Example 3.1.12]{Th-book}), and $\cG$-homeomorphisms are homeomorphisms that preserve the corresponding volume forms.

We are now ready to state the main theorem of this section, which implies Theorem~\ref{t-countable}.

\begin{thm} \label{t-main}
Let $\cG$ be a groupoid of germs of $\mathbb{R}^d$, and let $M$ be a closed $\cG$-manifold. Let $\Gamma$ be a finitely generated group with property $\fw$. 
For every homomorphism $\rho\colon \Gamma \to \Omega \mathsf{Homeo}_\cG(M)$,  one of the following holds:
\begin{enumerate}[(i)]
\item the action of $\rho(\Gamma)$ on $M$ has a  finite orbit;
\item \label{i-smooth2} there exist a closed $\cG$-manifold $N$ and an $\Omega\cG$-homeomorphism $\varphi \colon M \to N$ such that $\varphi\rho(\Gamma)\varphi^{-1}\subset \mathsf{Homeo}_\cG(N) $. 

\end{enumerate}
If moreover  $\cG=\mathcal D^r$ (so that $M$ is a differentiable manifold) and if the dimension $d$  is such that the euclidean sphere $\mathbb{S}^d$ is the unique exotic $d$-sphere up to diffeomorphism, then in part \eqref{i-smooth2} one can choose $N=M$.
\end{thm}
\begin{rem} \label{r-YvesCornulier}
In  the independent work \cite{CorPL}, Cornulier obtains results in similar spirit but in a different setting, namely by considering partially defined actions where the action of  each element is defined on the complement of a finite set and preserves the $\cG$-structure on its domain of definition. 
The main differences are that \cite{CorPL} only treats actions with finitely many singularities (here we allow the presence of countably many), but in dimension one the results of  \cite{CorPL} can be applied to discontinuous actions with finitely many discontinuity points (we only treat actions that are globally continuous). 
 
\end{rem}

In the proof we will use the following lemma, which we obtain by an elementary cut-and-paste argument. 

\begin{lem}\label{l-surgery}
Let $M$ be a   $\cG$-manifold of dimension $d$, and $x_1,\ldots, x_\ell$ be distinct points in $M$. Assume that $U_1,\ldots, U_\ell$ are pairwise
disjoint open neighbourhoods of $x_1,\ldots, x_\ell$ respectively, and that for every $i=1,\ldots, \ell$ we are given an $\Omega\cG$-homeomorphism $h_i\colon U_i\to V_i$, where $V_i$ is a $\cG$-manifold.
Then there exists a  $\cG$-manifold $N$ such that the following hold.
\begin{enumerate}
\item \label{i-homeo} As topological manifolds, $N$ and $M$ are homeomorphic. 
\item \label{i-commut} There exists an $\Omega\cG$-homeomorphism 
$\tau \colon M\to N$ such that for  every $i=1,\ldots, \ell$ there  exist an  open neighbourhood $W_i\subset U_i$ of $x_i$, 
and $\cG$-homeomorphic open embedding $\lambda_i\colon h_i(W_i)\hookrightarrow N$ such that 
\begin{enumerate}[(i)]
\item the set of singular points of $\tau$ is a compact subset of $W_1\cup\cdots\cup W_\ell$;
\item for  every $i=1,\ldots, \ell$, the restriction $\tau|_{W_i}$ coincides with $\lambda_i\circ h_i|_{W_i}$. 
\end{enumerate}
\item \label{i-sum} Assume further $\cG\subset \mathcal{D}^r$ for some $r\ge 1$  (in particular  $M$ and $N$ also inherit a structure of differentiable manifolds). Then, seen as a differentiable manifold, $N$ is diffeomorphic to a connected sum of $M$ with finitely many exotic spheres. In particular if the dimension $d$ is such that  the standard sphere $\mathbb{S}^d$ is the unique exotic $d$-sphere up to diffeomorphism, then $N$ is diffeomorphic to $M$. 
\end{enumerate}
\end{lem}

For a definition of the connected sum of differentiable manifolds and its basic properties, we refer the reader to \cite[\S 1]{Ker-Mil}. 

\begin{proof}[Proof of Lemma \ref{l-surgery}]
  Since every $h_i$ is a $\cG$-homeomorphism on the complement of a countable closed subset of $U_i$, using compactness we can choose disjoint neighbourhoods $W_i\subset U_i$ of the points $x_i$ such that the following properties are satisfied for every $i=1,\ldots,\ell$:
\begin{enumerate}[(a)]
\item \label{i-ball} the closure $\overline{W}_i$ is a $\cG$-manifold with boundary, topologically homeomorphic to a closed ball and is contained in~$U_i$;
\item $h_i$ is a $\cG$-homeomorphism in restriction to some open neighbourhood $L_i$ of $\partial \overline{W}_i$ in $U_i$. 
\end{enumerate}
Set $Z_i=h_i(W_i)$, and let $\overline{Z}_i$ be its closure in $V_i$, for $i=1,\ldots,\ell$. 
Consider the open subset $M_0=M\setminus\left(\overline{W}_1\cup\cdots\cup \overline{W}_\ell\right)$ of $M$, and let $\overline{M}_0$ be its closure in $M$. We have $\partial \overline{M}_0=\partial \overline{W}_1\cup\cdots \cup \partial\overline{W}_\ell$. 
Let $N$ be  the space  obtained by gluing $\overline{M}_0$ to $\overline{Z}_1\sqcup \cdots \sqcup \overline{Z}_\ell$, by identifying every  $\partial \overline{W}_i$  with  $\partial \overline{Z}_i$ via the map $h_i$, for every $i=1,\ldots, \ell$.  We denote by $\lambda_0\colon \overline{M}_0\hookrightarrow N$ and $\lambda_i\colon \overline{Z}_i\hookrightarrow N$, for $i=1,\ldots, \ell$, the canonical inclusions. Define $\tau\colon M\to N$  by 
 \[\tau(x)=\left\{\begin{array}{lr} \lambda_0(x), &\text{if } x\in M_0, \\
 \\
  \lambda_i \circ h_i(x), &\text{if }x\in \overline{W}_i, \text{ with } i=1,\ldots, \ell.\end{array}\right.\]
It is readily checked that $\tau$ is a homeomorphism from $M$ to $N$, in particular $M$ and $N$ are homeomorphic as topological manifolds.
Let us endow $N$ with the structure of  a $\cG$-manifold as follows. Write $L_i=L_i^+\cup L_i^-$, where $L_i^+=L_i\setminus W_i$ and $L_i^-=L_i\cap \overline{W}_i$, so that $L_i^+\subset \overline{M}_0$, $L_i^-\subset \overline{W}_i$, and $L_i^+\cap L_i^-=\partial \overline{W}_i$.  
The subset $K_i=\tau(L_i)$ is an open neighbourhood in $N$ of $\partial \overline{W}_i\simeq_{h_i} \partial\overline{Z}_i$. We transport the structure of  $\cG$-manifold of $L_i$ to $K_i$ via the homeomorphism $\tau\vert_{L_i}:L_i\to K_i$, for $i=1,\ldots,\ell$.
Moreover we transport the $\cG$-manifold structures of $M_0, Z_1,\ldots, Z_\ell$ to $\lambda_0(M_0), \lambda_1(Z_1),\ldots, \lambda_\ell(Z_\ell)$ via the identifications $\lambda_0,\lambda_1,\ldots,\lambda_\ell$. Note that the corresponding structures are compatible on the intersections $K_i$, because every $h_i$ is a $\cG$-homeomorphism on the interior of $L_i^{-}$. Therefore this defines uniquely the structure of $\cG$-manifold on $N$. By construction, the homeomorphism~$\tau$ defined above satisfies the conclusion of part \ref{i-commut}. 

 It remains to check part \ref{i-sum}. To this end, recall that by \eqref{i-ball} the sets $W_i$ above are such that $\overline{W}_i$ is a differentiable manifold with boundary homeomorphic to a closed ball.
 For $i=1,\ldots, \ell$,  let $R_i$ be the differentiable manifold obtained by gluing $\overline{W}_i$ and $\overline{Z}_i$ by identifying $\partial \overline{W}_i$ with $\partial \overline{Z}_i$ using the diffeomorphism $h_i\vert_{\partial \overline W_i}$. Since  both $\overline{W}_i$ and $\overline{Z}_i$ are homeomorphic to disks, and $R_i$ is an exotic sphere, it readily follows from the construction of  $N$ that it is diffeomorphic to a connected sum of $M$ with $R_1,\ldots, R_\ell$. The last sentence in part~\ref{i-sum} follows from the fact that  taking a connected sum with standard spheres does not change the diffeomorphism class of a manifold.
\end{proof}

Given a closed $\cG$-manifold  $M$ we denote by $\mathcal \cG_M$  the groupoid consisting of all germs of $\cG$-homeomorphisms between open subsets of $M$.
  Observe that, keeping the notations introduced in the previous section, we have $F(\cG_M)=\homeo_\cG(M)$ and $F(\Omega \cG_M)=\Omega\mathrm{Homeo}_{\cG}(M)$. The groupoid $\mathcal \cG_M$, however, needs not satisfy the requirements for Proposition \ref{p-countable}.
In order to get around this, we will instead view it as a subgroupoid of a larger groupoid, consisting of all germs of $\cG$-homeomorphims between all $\cG$-manifolds of dimension $d$. 

More precisely, let $\mathcal{B}_d$ be a set of closed compact $\cG$-manifolds of dimension $d$ that contains infinitely many representatives for each $\cG$-homeomorphism class (one representative would be enough, but taking infinitely many will slightly simplify the discussion). We define the space $\Xbb_d$ to be the disjoint union  $\Xbb_d=\bigsqcup_{M\in \mathcal{B}_d } M$. We   let $\cG_{\Xbb_d}$ be the groupoid of germs over the space $\Xbb_d$ consisting of all germs of all $\cG$-homeomorphisms  defined from an open subset of some manifold in $\mathcal{B}_d$ to an open subset of some (perhaps different) manifold in $\mathcal{B}_d$. 

For a fixed  $\cG$-manifold $M\in \mathcal{B}_d$, the group $\Omega \homeo_\cG(M)$ is naturally a subgroup  of $F(\Omega \mathcal \cG_{\Xbb_d})$ supported on $M$. In particular, it acts on the groupoid $\Omega \cG_{\Xbb_d}$ by post-composition whenever the target of the element of the groupoid lies in $M$,
and as the identity otherwise.

\begin{prop}\label{p-diffeo-resol}
The groupoid $\cG_{\Xbb_d}$ has countably resolvable singularities.
\end{prop}
\begin{proof}
To simplify the notations set $\cH=\cG_{\Xbb_d}$ and let $\cH_\alpha$ be the transfinite sequence as in the previous section. Assume that $\alpha=\beta+1$ is a successor, and let us show that the pair $\cH_\beta\subset \cH_\alpha$ has resolvable singularities (the case of a limit ordinal is totally analogous and we omit it). To this end take germs $\gamma_1,\ldots, \gamma_\ell \in \cH_\alpha$ and 
let $K\subset \Xbb_d$ be a compact subset as in Definition \ref{d-extension}. Setting $x_i=t(\gamma_i)$, for $i=1,\ldots, \ell$, we can assume that the points $x_1,\ldots, x_\ell$ belong to a closed $\cG$-manifold $M\subset \Xbb_d$ (equal to the union of all the connected $\cG$-manifolds that contain $x_1,\ldots, x_\ell$). Choose representatives $(h_1, x_1),\ldots, (h_\ell, x_\ell)$ of $\gamma_1^{-1}, \ldots, \gamma_\ell^{-1}$ whose domains of definitions are open neighbourhoods $U_1,\ldots, U_\ell$ of $x_1,\ldots, x_\ell$ respectively, that we can assume to be pairwise disjoint and contained in $M$. Moreover, by definition of $\cH_\alpha=S\cH_\beta$, upon reducing the neighbourhoods $U_i$  we can assume that $[h_i]_y\in \cH_\beta$ for every point $y\in U_i\setminus\{x_i\}$. 

We apply Lemma \ref{l-surgery} to this data; we let $N$ be the $\cG$-manifold and $\tau\colon M\to N$ and $W_i\subset U_i$ and $\lambda_i\colon h_i(W_i)\hookrightarrow N$ be as in the conclusion of the lemma. Without loss of generality, we can suppose that $N$ belongs to the set of manifolds $ \mathcal{B}_d$, so that $N\subset \Xbb_d$. Moreover, since we assume that  $ \mathcal{B}_d$ contains infinitely many representatives of every $\cG$-homeomorphism class, we can assume that $N\cap K=\varnothing$ and that $N\cap M=\varnothing$ in $\Xbb_d$.
Let $\varphi\in F(\cH_\alpha)$ be the element given by:
 \[\varphi(x)=\left\{\begin{array}{lr} \tau(x),& \text{if }x\in M, \\ \tau^{-1}(x), & \text{if }x\in N, \\ x,& \text{if }x\notin M\cup N. \end{array}\right.\]
We claim that the homeomorphism $\varphi$ satisfies the conclusions (\ref{i-resolvable},\ref{ii-resolvable}) of Definition \ref{d-extension}. To see this, observe that for every $i=1,\ldots, \ell$ we have
\[[\varphi]_{x_i}\gamma_i=[\lambda_i]_{h_i(x_i)}[h_i]_{x_i}\gamma_i=[\lambda_i]_{h_i(x_i)}\gamma_i^{-1}\gamma_i=[\lambda_i]_{h_i(x_i)}\in \cH\subset \cH_\beta.\]
In a similar way one checks that for every point $x\in M\setminus \{x_1,\ldots, x_\ell\}$ we have $[\varphi]_x\in \cH_\beta$. Since $\varphi$ is supported on $M\cup N$ and $N\cap K=\varnothing$, this shows that $[\varphi]_x\in \cH_\beta$ for every point $x\in K\setminus\{x_1,\ldots, x_\ell\}$. This concludes the proof. 
\end{proof}

Note that in the above proof, we have not used part \ref{i-sum} of Lemma \ref{l-surgery}. Taking it into account, the same proof yields the following when $\cG=\mathcal D^r, r\ge 1$. (In the statement, $\mathcal D^r_M$ denotes the groupoid consisting of all germs of $C^r$ diffeomorphisms between open subsets of $M$.)

\begin{prop}\label{p-low-dim} \label{l-Dr}
Let $M$ be a closed   differentiable manifold of dimension $d$, and assume that $d$ is such that the standard sphere $\mathbb{S}^d$ is the unique standard exotic $d$-sphere up to diffeomorphism. Then the groupoid $\mathcal{D}^r_M$ has countably resolvable singularities.
\end{prop}

\begin{proof}[Proof of Theorem \ref{t-main}]
Let $M$ be a closed manifold and  $\rho\colon \Gamma \to \Omega \mathsf{Diff}^r(M)$ be as in the statement, and assume that $\Gamma$ has no finite orbit. Since property $\fw$ passes to quotients, we can identify $\Gamma$ with its image. We identify $M$ with a subset of $\Xbb_d$, where $d$ is the dimension of $M$, and $\Gamma$ with a subgroup of $F(\Omega \mathcal D^r_d)$ supported on $M$. By Proposition~\ref{p-diffeo-resol}, we are  in position to apply Proposition~\ref{p-countable} and we obtain that either $\Gamma$ has a finite orbit in $M$, or that there exists an element $\varphi\in F(\Omega \mathcal D^r_d)$ such that $\varphi \Gamma \varphi^{-1}\subset F(\mathcal D^r_d)$. Note that the conjugate $\varphi \Gamma \varphi^{-1}$ is supported on the manifold $N=\varphi(M)$, which is the manifold in the statement. If the dimension $d$ is as in the last sentence, then we may repeat the same reasoning using the groupoid $\mathcal{D}^r_M$ by Proposition \ref{p-low-dim} and we obtain that $M=N$. 
\end{proof}

If we restrict the attention to actions with \emph{finitely many} singularities, the conclusion can be made more precise in the finite-orbit case.  Given a $\cG$-manifold $M$ we denote $S\homeo_\cG(M)=F(S\mathcal \cG_M)$ the group of finitely singular  $\cG$-homeomorphisms of $M$, i.e.~those that are $\cG$-homeomorphisms in restriction to the complement of a finite subset. It is naturally a subgroup of the topological full group $F(S\mathcal \cG_{\Xbb_d})$  (with the notations of the previous section). It follows from Proposition \ref{p-diffeo-resol} that  the pair $S\mathcal \cG_d\supset \mathcal \cG_d$ satisfies the requirements for Proposition~\ref{p-pair}. Using this, one obtains:

\begin{thm} \label{t-finitelymany}
	Let $M$ be a closed  $\cG$-manifold. Let $\Gamma$ be a finitely generated group with property $\fw$. 
For every homomorphism $\rho\colon \Gamma \to S \mathsf{Homeo}_\cG(M)$,  there exists a  closed $\cG$-manifold $N$ and a finitely singular $\cG$-homeomorphism $\varphi\colon M \to N$  such that the set of singular points of $\varphi \rho(\Gamma)\varphi^{-1}$ is finite and consists of points with a finite $\varphi \rho(\Gamma)\varphi^{-1}$-orbit.
\end{thm}

Again, if $\cG=\mathcal D^r$ and if the dimension $d$ is such that $\mathbb{S}^d$ is the unique exotic $d$-sphere up to diffeomorphism, then we can choose $M=N$. 

Let us now explain the applications related to Zimmer's conjecture.
\begin{cor} \label{c-Zimmer-vol}
	Let $G$ be a connected Lie group of real rank $r\ge 2$, whose Lie algebra is simple and with finite centre. Let $\Gamma\subset G$ be a cocompact lattice, or $\Gamma=\mathsf{SL}(r+1,\Z)$.
	Let $M^d$ be a closed manifold of dimension $d$, and let $\rho:\Gamma\to \Omega\mathsf{Diff}^2(M)$. 
	\begin{enumerate}[(i)]
\item \label{i-Zimmer1} If $d<r$, then the action of $\rho(\Gamma)$ on $M$ has a finite orbit.
\item \label{i-Zimmer2} If there is a volume form $\omega$ on $M$ such that $\rho(\Gamma)$ preserves the measure associated to $\omega$, and if $d\le r$, then $\rho(\Gamma)$ has a finite orbit. 
	\end{enumerate}
\end{cor} 
\begin{rem}
Note that the condition that $\rho(\Gamma)$ preserve the measure associated to $\omega$ is equivalent to say that for every element $\gamma\in \Gamma$, the homeomorphism $\rho(\gamma)$ preserves the form $\omega$ outside its singular set (as the singular set of $\rho(\gamma)$ is countable, it has measure zero).
\end{rem}
\begin{proof}[Proof of Corollary \ref{c-Zimmer-vol}]
Both facts follow by combining the main results in \cite{BFH, BFH2} with  Theorem \ref{t-main}, applied to the groupoids $\cG=\mathcal D ^2$ and $\cG=\mathcal D^2_{\mathsf{Leb}}$ respectively. 
\end{proof}

\begin{rem}
	Relying on Zimmer's work \cite{zimmer_T}, we can also obtain rigidity results for (ergodic) volume-preserving actions of Kazhdan's groups on closed $\mathcal G$-manifolds, for \emph{finite type} (or \emph{rigid}) $\mathcal G$-structures, such as a pseudo-Riemannian metric. {Moreover, as in \cite{zimmer_T} or \cite{navas-cocycles}, it could be interesting to obtain a version of Theorem~\ref{t-main} for \emph{cocycles}. See also the discussion in \cite[\S\S4-6]{fisher_rigidity}.}
\end{rem}

 \section{Singular actions on the circle} \label{s:circle}
In this section  we analyse in more detail various classes of singular actions on the circle.

\subsection{Piecewise differentiable homeomorphisms}

\begin{dfn}\label{d-piecewise}
A homeomorphism $g\colon \T\to \T$ is said to be \emph{piecewise differentiable} of class $C^r$ if there exists a finite subset $\BP(g)\subset  \T$ such that $g$ is a $C^r$-diffeomorphism in restriction to $\T\setminus \BP(g)$, and for every point $x\in \BP(g)$ the right and left derivatives of $g$ exist up to order $\lfloor r\rfloor$ at $x$, and the derivative $D^{(\lfloor r \rfloor)}g$ is  $(r-\lfloor r \rfloor)$-H\"older continuous. The subset $\BP(g)$ will be called the set of \emph{breakpoints} of $g$. The group of all piecewise differentiable homeomorphisms of $\T$ of class $C^r$ is denoted by $\pa$.  \end{dfn}

In the following, we will simply denote by $\mathcal D^r$ the groupoid of diffeomorphic germs $\mathcal D^r_{\T}$ over $\T$, and by $\mathcal{PD}^r$ the groupoid of piecewise-$C^r$ germs over $\T$. The difference between $S\mathcal D^r$ and $\mathcal{PD}^r$ is that in the latter case we require that every germ is represented by a local homeomorphism having left and right first $\lfloor r\rfloor$ derivatives defined at every point of the domain of definition, and with the $\lfloor r\rfloor$-th left and right derivatives $(r-\lfloor r\rfloor)$-H\"older continuous. Observe that  $\mathcal {PD}^r\supset \mathcal D^r$ is a singular pair of groupoids, and it is not difficult to show that it has resolvable singularities.

We let $\mathcal{PD}^r_{+}$ be the   subgroupoid consisting of \emph{orientation preserving} germs. For every point $x\in \T$ we let $(\mathcal{PD}^r_{+})_x$ be the  \emph{isotropy group} at $x$ of this groupoid, which consists of all germs $\gamma \in \mathcal{PD}^r_{+}$ such that $s(\gamma)=t(\gamma)=x$. We first settle the following special case of Theorem \ref{t-piecewise}.

\begin{prop}\label{p:globalfixed}
	Let $\G$ be a countable property $\fw$ subgroup of $\mathsf{PDiff}^1(\T)$ that has a finite orbit. Then $\G$ is finite. 
\end{prop}

\begin{proof}
Assume by contradiction that $\G$ is infinite, and note that it is finitely generated because of property $\fw$.  Since property $\fw$ passes to finite index subgroups, we can assume  that $\G$  preserves orientation and fixes a point $x\in \T$, and using that $\G$ is finitely generated, without loss of generality we can assume that there exists an element $g\in \G$ such that $g$ is not the identity in restriction to a right neighbourhood of $x$. Hence we obtain a non-trivial group homomorphism $\G\to( \mathcal{PD}^1_+)_{x}$. The image of this homomorphism is a non-trivial finitely generated subgroup of $( \mathcal{PD}^1_+)_{x}$, and therefore admits a non-trivial homomorphism to~$\R$ by the Thurston's stability theorem \cite{Th}. Hence so does~$\G$. This contradicts property $\fw$, since $\fw$ passes to quotients, and an abelian infinite  finitely generated group never has $\fw$. 
\end{proof}

\begin{thm}\label{t-piecewise}
Fix $r\in [1,\infty]$. Every countable subgroup of $\pa$ with property $\fw$  is conjugate in $\pa$ to a subgroup of  $\dg(\T)$.
\end{thm}
\begin{proof}
Let $\Gamma\subset \pa$ be a subgroup with property FW. Since the singular pair $\mathcal{PD}^r\subset \mathcal D^r$ has resolvable singularities, we can invoke Proposition~\ref{p-pair} and claim that the action of $\Gamma$ on $\T$ has a finite orbit, or $\Gamma$ is conjugate in $F(S\mathcal D^r)=\pa$ to a subgroup of $F(\mathcal D^r)=\dg(\T)$. However, because of Proposition~\ref{p:globalfixed}, the former is possible only if $\Gamma$ is finite.
\end{proof}

\begin{proof}[Proof of Corollary~\ref{c-T}]
	Let $\G$ be a subgroup of $\mathsf{PDiff}^{3/2}(\T)$,  with property $(T)$ and hence property $\fw$. (We again identify $\Gamma$ with its image.) After Theorem~\ref{t-piecewise}, there exists a homeomorphism $\varphi:\T\to\T$ such that the conjugate $\varphi \G\varphi^{-1}$ is a group of $C^{r}$ circle diffeomorphisms. Navas's theorem (Theorem~\ref{t:navas}) gives that $\G$ is finite.
\end{proof}

\subsection{PL circle homeomorphisms}\label{s:PL}
\renewcommand{\cH}{\mathcal H}

In this subsection we will give an alternative and more elementary proof of Corollary~\ref{c-T} for groups of piecewise linear homeomorphisms, that we state below. For simplicity, we will always assume that homeomorphisms do preserve the orientation.

\begin{dfn}
A homeomorphism $h:\T\to\T$ is \emph{piecewise linear} if, when seeing $\T$ as the flat torus $\R/\Z$ (so with its quotient affine structure), for all but finitely many points $x\in \T$ there exists an open neighbourhood $I(x)$ such that the restriction $h\vert_{I(x)}$ is an affine map, that is of the form $y\mapsto ay+b$. A point $x\in \T$ where this condition is not verified is a \emph{breakpoint} of $h$. We write $\BP(h)$ for the set of breakpoints of $h$.
We denote by $\PL(\T)$ the group of piecewise linear homeomorphisms of $\T$.
\end{dfn}

\begin{thm}\label{t-PL}
	If $\Gamma$ is a countable Kazhdan group, every homeomorphism $\rho:\Gamma\to \PL(\T)$ has finite image.
\end{thm}

We consider the Hilbert space $\cH=\ell^2(\T)$, that is, the space of square-summable families $(a_x)_{x\in \T}\subset \R$, indexed by points of $\T$ (note that this forces $a_x= 0$ for all but countably many indices $x\in \T$). If a group $\G$ acts on $\T$, the action induces an isometric action on $\cH$, just by permutation of the basis: $h\cdot (a_x)_{x\in \T}= \left (a_{h^{-1}(x)}\right )_{x\in \T}$. We write $\pi:\G\to \mathsf{O}(\cH)$ for this (left) representation.
Given an element $h\in \PL(\T)$, we write $D^-h$ and $D^+h$ for its left and right derivatives (which are well-defined at every point). The map
\begin{equation*}\label{eq:jump}
\dfcn{b}{\PL(\T)\times \T}{\R}{(h,x)}{\log \frac{D^+h(x)}{D^-h(x)}}
\end{equation*}
is a cocycle (i.e.~$b(gh,x)=b(g,h(x))+b(h,x)$), as one sees by applying the chain rule. Observe that we have $b(h^{-1},x)=-b(h,h^{-1}(x))$. As $h$ is a piecewise linear homeomorphism, for fixed $h$, the function $b(h,-)$ is zero at all but finitely many points (the breakpoints of $h$). This allows to define an isometric affine action $\rho:\G\to \Isom(\cH)$, with linear part~$\pi$ and translation part $b$:

\begin{lem}\label{l:action}
Let $\G$ be a subgroup of $\PL(\T)$.
The map $\rho$, defined for $h\in \G$ and $(a_x)\in \ell^2(\T)$ by
\begin{align*}
\rho(h)(a_x)_{x\in\T}:=&\,\pi(h)(a_x)_{x\in\T}+(b(h^{-1},x))_{x\in\T}\\
=&\,\left (a_{h^{-1}(x)}+b(h^{-1},x)\right )_{x\in \T},
\end{align*}
defines a homomorphism $\rho:\G\to \Isom(\cH)$.
\end{lem}

\begin{proof}
One needs a few simple verifications.
First, we check that the map $\rho$ actually defines a homomorphism. This is guaranteed by the fact that $b$ is a cocycle (over $\pi$):
\begin{align*}
\rho(g)\rho(h)(a_x)_{x\in\T} & =\rho(g)\left (a_{h^{-1}(x)}+b(h^{-1},x)\right )_{x\in \T}\\
& = \left (a_{h^{-1}g^{-1}(x)}+b(h^{-1},g^{-1}(x))+b(g^{-1},x)\right )_{x\in \T}\\
& = \left (a_{(gh)^{-1}(x)}+b(h^{-1}g^{-1},x)\right )_{x\in \T}= \rho(gh)(a_x)_{x\in\T}.
\end{align*}
 Secondly, we have to verify that if a sequence $(a_x)_{x\in\T}$ is in $\cH$, then also its $\rho(h)$-image is. This is because for a fixed element $h\in\PL(\T)$, the cocycle $b(h^{-1},-)$ is zero at all but finitely many points. Finally, we have to verify that for fixed $h\in\PL(\T)$, the map $\rho(h)$ is an isometry. This is because $\rho(h)$ is the composition of the isometry $\pi(h)$ with the translation $b(h^{-1},-)$.
\end{proof}

\begin{lem}\label{l:liousse}
	Let $f\in \PL(\T)$ be a non-trivial element acting with a fixed point on $\T$, then $\rho(f):\cH\to\cH$ has an unbounded orbit.
\end{lem}

\begin{proof} We reproduce an argument appearing in \cite[\S C]{LiousseFourier}.
		Let $I=(x_0,x_1)$ be a connected component of the open support $\{x\in\T\colon f(x)\neq x\}$ (possibly $x_0=x_1$).
	Without loss of generality, we can assume that $f$ is contracting on~$I$, i.e.~$f(y)<y$ for every $y\in I$.
	
	\begin{claim}
		The set of values $S=\{b(f^n,x)\colon n\in \N, x\in I\}$ is finite.
	\end{claim}
	\begin{proof}[Proof of claim]
	Choose points $x_0<z<x_1$ such that the restriction of $f$ to $I_0=(x_0,z)$ is linear.
	Given a point $x\in I$ and $n\in\N$, the logarithm of the jump of derivatives $b(f^n,x)$ can be written as the sum
	\[
	b(f^n,x)=\sum_{k=0}^{n-1}b(f,f^k(x)).
	\]
	Notice that only points $f^k(x)$ that are breakpoints of $f$ contribute to the sum above.
	As $f$ is a contraction, the sequence $\{f^k(x)\}_{k\in\N}$ is strictly decreasing to $x_0$, so the forward orbit of $x$ visits any breakpoint of $f$ at most once. Therefore $b(f^n,x)$ can only take values in the finite set $\left\{
	\sum_{c\in E}b(f,c)\right\}_{E\subset\BP(f)}$.
	\end{proof}
	
	
	{ \begin{claim}\label{c:bp}
			Let $M_n$ denote the number of breakpoints of $f^n$. Then $M_n$ goes to infinity as $n\to\infty$.
		\end{claim}
	\begin{proof}[Proof of claim]
		Recall that $f$ is contracting on $I=(x_0,x_1)$. Thus
		$\lim_{n\to+\infty} D^+f^n(x_0)=0$ and $\lim_{n\to+\infty}D^-f^n(x_1)=+\infty$.  As the derivative of $f^n$ can increase only when passing through breakpoints, and at breakpoints it changes by a uniformly bounded amount, we get easily the claim. Indeed, setting $c_0=\log D^+f(x_0)<0$, $c_1=\log D^-f(x_1)>0$ and $\mu=\max S$, one must have $M_n\ge n\frac{c_1-c_0}{\mu}$.
	\end{proof}	
	}
	
	 Take $\beta=\min_{\sigma\in S\setminus\{0\}}|\sigma|$. The image $\rho(f^{n})\vec0$ of the vector $\vec0=(0)_{x\in\T}$ is exactly the vector $\left (b(f^{-n},x)\right )_{x\in \T}$ whose squared $\ell^2$-norm is at least $M_n\beta^2$. As the sequence $M_n$ is unbounded, we get that the sequence $\|\rho(f^n)\vec0\|$ is unbounded.
\end{proof}

\begin{rem}
	In the proof of Claim \ref{c:bp}, the fact that elements are piecewise linear is fundamental for the proof. Indeed, every non-trivial piecewise linear homeomorphism of an interval has at least one breakpoint. This is no longer true, even for M\"obius transformations.
\end{rem}

We next recall a classical result due to H\"older \cite[Theorem 2.2.32]{navas-book}. For its statement, recall that a group action is \emph{free} if every point has trivial stabiliser.

\begin{thm}[H\"older]\label{t:Holder}
	Let $\G$ be a subgroup of $\Homeo_+(\T)$, whose action on $\T$ is free. Then~$\G$ is isomorphic to a group of rotations. More precisely, if $\Phi:\G\to \Phi(\G)\subset \SO(2)$ denotes the isomorphism, there exists a monotone continuous degree $1$ map $h:\T\to \T$ such that $hg=\Phi(g)h$ (i.e.~$\G$ is semi-conjugate to a group of rotations).
\end{thm}

We can now prove the main result of the subsection:
\begin{proof}[Proof of Theorem~\ref{t-PL}]
	As the quotient of a Kazhdan group is also a Kazhdan group, it is enough to prove that every countable subgroup of $\PL(\T)$ which has property $(T)$ is actually finite.
	So let $\G\subset \PL(\T)$ be an infinite countable subgroup. We consider the representation $\rho:\G\to\Isom(\cH)$ of Lemma~\ref{l:action}. 
	Even if property $(T)$ is not inherited by subgroups, we can however restrict to subgroups to show that a particular isometric affine action has unbounded orbits.
	If $\G$ contains a non-trivial element acting with a fixed point, then Lemma~\ref{l:liousse} implies that the representation $\rho$ has an unbounded orbit, disproving property $(T)$. Otherwise, Theorem~\ref{t:Holder} implies that $\G$ is isomorphic to a group of rotations (in fact, by Denjoy's theorem, which holds for piecewise linear homeomorphisms~\cite{herman}, it is topologically conjugate), which is abelian, and hence cannot have property~$(T)$.
\end{proof}

A direct application of the work of Minakawa on \emph{exotic circles} of $\PL(\T)$ \cite{minakawa} (not to be confused with exotic spheres in the sense of differentiable topology)  gives a more precise statement for Corollary~\ref{c-T} in the PL case. For this, we set some notations. A \emph{topological circle} in $\PL(\T)$ is a one-parameter subgroup $S=\{g_\alpha;\alpha\in \T\}\subset \PL(\T)$ which is topologically conjugate to the group of rotations $\SO(2)$. The topological circle is called \emph{exotic}, if the conjugating map cannot be taken in $\PL(\T)$. 
For $A>1$ one sets $I_A=[1/(A-1),A/(A-1)]$, which is an interval of length $1$, and defines
\[
\dfcn{h}{I_A}{[0,1]}{x}{\frac{\log (A-1)x}{\log A}}
\]
which naturally extends to a homeomorphism $\widetilde{h_A}$ commuting with the translation by $1$ and hence defines, by quotient, a homeomorphism $h_A$ of $\T$. Set $S_A= h_A^{-1}\SO(2)h_A$. When $A<1$, we set $S_A = \iota S_{A^{-1}} \iota$, where $\iota:\T\to \T$ is the order-reversing involution defined by $\iota(x)=-x$. These are exotic circles. For $A=1$, the notation $S_A$ will simply stay for  the standard $\SO(2)$. The circles $S_A$ are contained in $\PL(\T)$ (these examples had previously appeared in \cite{herman,boshernitzan}) and Minakawa shows that every topological circle in $\PL(\T)$ is PL conjugate to one of the exotic circles $S_A$ (and vice versa). More precisely, he proves that the only irrational rotations for which 1) the number of breakpoints of iterates is bounded and 2) the set of jumps is finite, are contained in a topological circle (see also \cite{Liousse}). Relying on this, one can prove:

\begin{thm}\label{t:boundedorbits}
	Let $\G$ be an infinite countable subgroup of $\PL(\T)$. The following statements are equivalent:
	\begin{enumerate}
		\item the affine isometric action $\rho:\G\to\Isom(\cH)$ defined in Lemma \ref{l:action} has bounded orbits;
		\item the subgroup $\G$ is PL conjugate into a topological circle of rotations $S_A$.
	\end{enumerate}
\end{thm}

\begin{rem}
	Exotic circles were one of the main reasons for suspecting that $\PL(\T)$ might have contained subgroups with property $(T)$ (see the discussion in \cite{BLT}).
	The statement in Theorem~\ref{t:boundedorbits} clearly gives a negative answer, and at the same time explains the role of exotic circles in this problem.
\end{rem}

\footnotesize{\paragraph*{Acknowledgements}
M.T.~thanks Christian Bonatti for the helpful discussion about germs of singular diffeomorphisms. N.M.B.~thanks Peter Feller for useful clarifications about exotic spheres. The authors are grateful to Yves de Cornulier for sharing his work \cite{CorPL} and for the several remarks.}

\medskip

\noindent\textit{Yash Lodha\\
EPFL\\
SB MATH EGG\\
MA C3 584 (Batiment MA) Station 8\\
CH-1015 Lausanne\\
Switzerland\\}
\href{mailto:yash.lodha@epfl.ch}{yash.lodha@epfl.ch}

\medskip

\noindent\textit{Nicol\'as Matte Bon\\
CNRS\\
Institut Camille Jordan (ICJ, UMR CNRS 5208)\\
Universit\'e de Lyon\\
43 blvd.\ du 11 novembre 1918\\
69622 Villeurbanne\\
France\\}
\href{mailto:mattebon@math.univ-lyon1.fr}{mattebon@math.univ-lyon1.fr}

\medskip

\noindent\textit{Michele Triestino\\
Institut de Math\'ematiques de Bourgogne (IMB, UMR CNRS 5584)\\
Universit\'e Bourgogne Franche-Comt\'e\\
9 av.~Alain Savary\\
21000 Dijon\\
France\\}
\href{mailto:michele.triestino@u-bourgogne.fr}{michele.triestino@u-bourgogne.fr}

\end{document}